\documentclass[10pt]{amsart}

\usepackage{amsmath,amssymb,amsfonts,amsthm,mathtools}
\usepackage{enumitem}
\usepackage{hyperref}

\usepackage{graphicx}
\usepackage{tikz}
\usepackage{pgfplots}
\pgfplotsset{compat=1.17}

\usepackage[colorinlistoftodos]{todonotes}

\usepackage{kotex}

\numberwithin{equation}{section}

\hypersetup{
	colorlinks=true,
	linkcolor=blue,
	citecolor=red,
	urlcolor=cyan
}

\newtheorem{theorem}{Theorem}[section]
\newtheorem{proposition}[theorem]{Proposition}

\theoremstyle{definition}
\newtheorem{definition}[theorem]{Definition}

\theoremstyle{remark}
\newtheorem{remark}[theorem]{Remark}

\theoremstyle{remark}
\newtheorem{lemma}{Lemma}

\newcommand{\R}{\mathbb{R}}

\newcommand{\Z}{\mathbb{Z}}

\newcommand{\e}{\mathrm{e}}
\newcommand{\supp}{\operatorname{supp}}

\title[Sharp $L^2$ Estimates for binomial phases]
{Sharp $L^2$ Estimates for 
$(2+1)$-dimensional 
oscillatory integral operators with homogeneous
binomial phases}

\author{Chu-hee Cho, Jin Bong Lee, and Chan Woo Yang}
\address{Department of Mathematics \\
	Korea University \\
	Seoul 02841, Korea}
\email{akilus@korea.ac.kr}
\address{Research Institute of Mathematics\\
	Seoul National University \\
	Seoul 08826, Korea}
\email{jinblee@snu.ac.kr}
\address{Department of Mathematics \\
	Korea University \\
	Seoul 02841, Korea}
\email{cw$\_$yang@korea.ac.kr}
\date{}
\thanks{}

\subjclass[2020]{Primary 42B20; Secondary 35S30, 42B15}
\keywords{Oscillatory integral operators, degenerate phases, $L^2$ decay, sharp estimates}

\begin{document}
	
	\begin{abstract}
		
			We study oscillatory integral operators in $(2+1)$-dimensions with a homogeneous binomial phase
			\[
			\Phi(x,y,t)=x^{k-k_P}t^{k_P}+y^{k-k_Q}t^{k_Q},
			\qquad 1\le k_P<k_Q<k.
			\]
			For compactly supported smooth amplitudes, we establish sharp \(L^2(\R)\to L^2(\R^2)\) estimates with logarithmic losses occurring only in certain critical cases. The proof is based on scale-dependent Phong--Stein estimates.

	\end{abstract}
	
	\maketitle
	
	\section{Introduction}\label{sec:introduction}
    Consider oscillatory integral operators in $(d_u+d_v)$ dimensions
    \[
        T_\lambda f(u) = \int_{\mathbb R^{d_v}} \mathrm e^{i \lambda\Phi(u,v)} a(u,v) f(v)\,\mathrm dv, \quad u \in \mathbb R^{d_u}
    \]
    with $a\in C_c^\infty(\mathbb R^{d_u+d_v})$ and $\Phi$ is a smooth phase function.
    We are interested in the case of $d_u=2$ and $d_v=1$ with a specific phase function $\Phi$, i.e.
    \[
    T_\lambda f(x,y)=\int_{\mathbb R} e^{i\lambda \Phi(x,y,t)} a(x,y,t)f(t)\,\mathrm dt,
    \]
    where $a\in C_c^\infty(\mathbb R^3)$ and
    \[
    \Phi(x,y,t) = x^{k-k_P}t^{k_P} + y^{k-k_Q}t^{k_Q}, \qquad 1\le k_P<k_Q<k.
    \]
    One of fundamental problems in studying such operators is to determine the optimal rate of decay of $L^2(\mathbb R)\to L^2(\mathbb R^2)$ operator norm $\|T_\lambda\|$ as $\lambda\to\infty$.

    For nondegenerate phases in $(d_u+d_v)$ dimensions, the theory is well understood. Hörmander
    \cite{Ho73} proved that if the mixed Hessian of the phase is
    nondegenerate on the support of the amplitude, then the associated
    oscillatory integral operator satisfies the sharp
    $L^2$ decay estimate. This result was later extended to the setting of Fourier integral operators by Duistermaat and Hörmander
    \cite{Ho71,DH72}, and has become a cornerstone of the modern theory of oscillatory integral operators.

    The degenerate case is considerably more subtle. When the mixed Hessian becomes degenerate, the decay is no longer universal and depends delicately on the geometry of the phase. Over the past several decades, sharp estimates have been obtained for
    various classes of degenerate oscillatory integral operators,
    including operators associated with fold and cusp singularities in $(d_u+d_v)$ dimensions \cite{GS94,GS98,PS90}, Newton polyhedra in $(1+1)$ dimensions
    \cite{PS97,Va76},
    and polynomial phases in $(1+1)$ dimensions
    \cite{RS87,CCW99}.
    Despite these developments, obtaining sharp decay estimates for
    higher-dimensional degenerate models remains a challenging problem.

    For phases depending on more than two variables, obtaining sharp decay estimates
    is considerably more difficult. Greenleaf--Pramanik--Tang \cite{GPT07} studied
    oscillatory integral operators with homogeneous polynomial phases in several
    variables and established sharp estimates under suitable structural hypotheses. 
    The homogeneous binomial family considered in this paper lies outside of the scope of the structural hypothesis of \cite{GPT07}. One of goals of this paper is to provide a
    natural testing ground in which the interaction between the two monomial
    components can be analyzed explicitly. Its explicit structure allows one to identify the dominant scales and establish optimal $L^2$ bounds.

    In the special case of $(2+1)$ dimensions, several sharp results are known for homogeneous cubic phases. A classical family is
    \[
    (u,v,t)\longmapsto P_1(u,v)t^2+P_2(u,v)t,
    \]
    where $P_1$ and $P_2$ are homogeneous polynomials of degrees one and two,
    respectively. Tang \cite{Ta06} proved that if $P_2$ has no repeated linear
    factors over $\mathbb C$ and $P_1\not\equiv0$, then
    \[
    \|T_\lambda\|_{L^2(\mathbb R)\to L^2(\mathbb R^2)}
    \lesssim
    \lambda^{-1/2}
    (\log\lambda)^\gamma,
    \]
    for a suitable logarithmic exponent $\gamma$.
    In contrast, when repeated factors occur,
    the mixed Hessian degenerates more severely,
    leading to slower decay.

    For the following degenerate phase function
    \[
    S(u,v,t)=ut^2+v^2t,
    \]
    Tan--Xu \cite{TX24} established the slower decay estimate
    \[
    \|T_\lambda\|_{L^2\to L^2}
    \lesssim
    \lambda^{-3/8},
    \]
    and Lang--Tang \cite{LT25} subsequently proved that the exponent $3/8$ is sharp
    whenever the amplitude does not vanish at the origin. Xu \cite{Xu25}
    also obtained sharp $L^2\to L^p$ decay estimates for degenerate
    $(2+1)$-dimensional oscillatory integral operators, while the recent
    broad--narrow approach of \cite{Xu23} yields another proof of a related
    sharp decay estimate. These results demonstrate that even among homogeneous
    polynomial phases of fixed degree, the multiplicity of polynomial factors can
    change the optimal decay exponent.

    This motivates the study of the following homogeneous binomial family
    \begin{equation}\label{eq:intro_phase}
    \Phi(x,y,t)=x^{k-k_P}t^{k_P}+y^{k-k_Q}t^{k_Q},
    \qquad
    1\le k_P<k_Q<k.
    \end{equation}
    This family is homogeneous of degree $k$ with respect to the standard dilation
    \[
    \Phi(rx,ry,rt)=r^k \Phi(x,y,t),
    \qquad r>0,
    \]
    and exhibits a broad range of degenerate behaviors depending on the parameters
    $k_P$ and $k_Q$. Although the phase consists of only two monomials, the
    competition between the two terms produces several distinct scaling regimes,
    each leading to a different decay exponent. The principal difficulty is to determine which scaling regime governs the $L^2$ decay and to identify precisely the exceptional cases where logarithmic
    losses occur.

    Let
    \begin{equation}\label{eq:intro_operator}
    T_\lambda f(x,y)=
    \int_{\mathbb R}
    e^{i\lambda \Phi(x,y,t)}
    a(x,y,t)f(t)\,\mathrm dt,
    \end{equation}
    where $a\in C_c^\infty(\mathbb R^3)$ is supported in $[-1,1]^3$.
    Our main result determines the sharp
    $L^2(\mathbb R)\rightarrow L^2(\mathbb R^2)$
    decay of $T_\lambda$ for every choice of
    $(k,k_P,k_Q)$.
    The optimal decay exponent is expressed explicitly in terms of the parameters,
    and logarithmic losses occur only in two exceptional configurations.
    Moreover, we prove that these exponents are optimal by constructing Knapp-type lower bounds.

	To state our main theorem, we first introduce three candidate decay exponents.
    These correspond to the three possible scaling regimes arising in a dyadic
    decomposition of the operator.
	
	\begin{definition}[Decay exponents]\label{def:rho}
		Let \(k,k_P,k_Q\) satisfy \begin{equation}\label{eq:intro_parameters}
		k\ge 3,\qquad 1\le k_P<k_Q<k.	
	\end{equation} Define
		\begin{align*}
			\rho_1
			&=
			\frac1{2(k-k_P)}+\frac1{2(k-k_Q)},\\
			\rho_2
			&=
			\frac1{2(k-k_P)}
			+
			\frac1{2k_Q}\Bigl(1-\frac{k_P}{k-k_P}\Bigr)
			=
			\frac{k_Q+k-2k_P}{2k_Q(k-k_P)},\\
			\rho_3
			&=
			\frac1{2k_P},
		\end{align*}
		and set
		\begin{equation}\label{eq:rho_star}
			\rho_*=\min\{\rho_1,\rho_2,\rho_3\}.
		\end{equation}
		We also distinguish the two critical parameter configurations
		\begin{equation}\label{eq:borderlineA}
			1=\frac{k_P}{k-k_P}+\frac{k_Q}{k-k_Q},
		\end{equation}
		and
		\begin{equation}\label{eq:borderlineB}
			k=2k_P.
		\end{equation}
	\end{definition}

	The following theorem completely determines the sharp
    $L^2(\mathbb R)\rightarrow L^2(\mathbb R^2)$
    decay of the operator $T_\lambda$ associated with the phase
    \eqref{eq:intro_phase}. The power of $\lambda$ is always optimal, while logarithmic losses occur
    only in the two critical situations described above.
	\begin{theorem}\label{thm:main}
		Let \(k,k_P,k_Q\) satisfy \eqref{eq:intro_parameters}, and let \(T_\lambda\) be defined by
		\eqref{eq:intro_operator} with phase \eqref{eq:intro_phase}.  Then there exists \(C>0\) such that, for all
		\(\lambda\gg1\),
		\begin{equation}\label{eq:upperbound_main}
			\|T_\lambda\|_{2\to2}
			\le C\,\lambda^{-\rho_*}\,(\log\lambda)^{\gamma},
		\end{equation}
		where \(\rho_*\) is given by \eqref{eq:rho_star}, and the logarithmic exponent \(\gamma\in\{0,1\}\) is
		\begin{equation}\label{eq:gamma_def}
			\gamma=
			\begin{cases}
				1,& \text{if } k=2k_P,\\
				1,& \text{if } k\neq 2k_P,\ \rho_*=\rho_1,\ \text{and }\eqref{eq:borderlineA}\text{ holds},\\
				0,& \text{otherwise}.
			\end{cases}
		\end{equation}
	\end{theorem}
    For $k<2k_P$, the index $\rho_*$ is simplified to $\rho_3 = 1/2k_P$, hence the $L^2$ decay depends only on $k_P$.
	The proof of Theorem~\ref{thm:main} is based on dyadic localization with localized Phong--Stein estimates. Since the mixed Hessian varies according to the dyadic scales of $x$, $y$, and
    $t$, each localized piece satisfies a different oscillatory estimate.
    Optimizing these dyadic bounds yields the sharp decay estimate.

    \begin{remark}
    The sharp decay exponent $\rho_*$ admits a simple geometric interpretation.
    Let $d_{\mathcal N}$ denote the Newton distance of the phase $\Phi$.
    Then, we have the identity
    \[
    \rho_*=\frac1{2d_{\mathcal N}},
    \]
    which shows that the sharp decay exponent is determined by the Newton distance of the
    phase. This is consistent with the role of the Newton distance in the asymptotic analysis of scalar oscillatory integrals associated with Newton polyhedra (see \cite{PS97}).
    \end{remark}

    \begin{remark}

    The three candidate exponents correspond to three distinct scaling regimes in
    the dyadic decomposition.
    \begin{itemize}
    \item
    $\rho_1$ corresponds to the balanced regime, where both monomials contribute
    equally to the oscillation.
    
    \item
    $\rho_2$ corresponds to an intermediate regime in which oscillation is exploited
    essentially in only one output direction.
    
    \item
    $\rho_3$ corresponds to the size-dominated regime, where the support in the
    integration variable is too short for oscillation to produce additional decay.
    
    \end{itemize}
    The critical relations \eqref{eq:borderlineA} and \eqref{eq:borderlineB} correspond to borderline geometric
    series in the dyadic summation and therefore produce logarithmic losses.
    \end{remark}

    The upper bound in Theorem~\ref{thm:main} is optimal up to the logarithmic growth.
    The following theorem provides the corresponding lower bound.
	
	\begin{theorem}\label{thm:lbound}
		Let \(k,k_P,k_Q\) satisfy \eqref{eq:intro_parameters}, and let \(T_\lambda\) be defined by
		\eqref{eq:intro_operator} with phase \eqref{eq:intro_phase}.  Assume in addition that
		\begin{equation*}
			a(0,0,0)\neq 0.
		\end{equation*}
		Then there exist constants \(c>0\) and \(\lambda_1\ge2\) such that, for all \(\lambda\ge\lambda_1\),
		\begin{equation}\label{eq:lowerbound_main}
			\|T_\lambda\|_{2\to2}
			\ge c\,\lambda^{-\rho_*}.
		\end{equation}
	\end{theorem}

	\begin{remark}
    Theorem~\ref{thm:main} recovers several previously known sharp decay estimates.
    When
    \[
    (k,k_P,k_Q)=(3,1,2),
    \]
    the phase becomes
    \[
    S(x,y,t)=x^2t+yt^2.
    \]
    In this case
    \[
    \rho_*=\frac38,
    \]
    which agrees with the sharp decay established by
    Tan--Xu \cite{TX24} and Lang--Tang \cite{LT25}.
    Our theorem therefore extends the known cubic result to the entire homogeneous
    binomial family.
	\end{remark}

	\subsection*{Organization}
	The paper is organized as follows.
    Section~\ref{sec:pre} introduces the dyadic decomposition and recalls the localized Phong--Stein estimate used throughout the paper. Section~\ref{sec:thm_main} reduces Theorem~\ref{thm:main} to a discrete optimization problem involving a triple dyadic sum.
    Section~\ref{sec:lower} establishes the lower bound in Theorem~\ref{thm:lbound} by a scale-adapted Knapp-type construction.
    Finally, Section~\ref{sec:triplesum} computes the triple sum and completes the proof of the main
    theorem.

	\subsection*{Notation}\label{subsec:conv}
	
	We write \(A\lesssim B\) to mean \(A\le C\,B\) with a constant \(C>0\) independent of \(\lambda\)
	and of the dyadic indices, where the implicit constant may depend on \(k,k_P,k_Q\) and finitely many seminorms of \(a\).
	We write \(A\approx B\) if both \(A\lesssim B\) and \(B\lesssim A\) hold.
	All logarithms are natural; replacing \(\log\) by \(\log_2\) only changes constants.

	\section{Preliminaries}\label{sec:pre}
	
	The proof of Theorem~\ref{thm:main} relies on two ingredients.
    The first is a localized $(1+1)$-dimensional oscillatory integral estimate of Phong--Stein type, while the second is a dyadic decomposition adapted to the small-support setting.
    This section collects these preliminary tools.
	
	\subsection{A $(1+1)$-dimensional $L^2$ lemma of Phong--Stein}\label{subsec:PS}
	
	We repeatedly use the following localized version of the
    Phong--Stein estimate (cf. \cite[Lemma~ 1.1]{PS94}).
    Compared with the original statement,
    we explicitly formulate the assumptions needed in our dyadic setting,
    namely the support condition in the integration variable and uniform
    $t$-derivative bounds.
	
	\begin{lemma}\label{lem:PS}
		Let a $(2+1)$-dimensional oscillatory integral operator be given by
		\begin{equation}\label{eq:PS_op}
			\mathcal{T}_\lambda g(x)
			=
			\int_{\R}\e^{\,i\lambda \mathcal S(x,t)}\,\psi(x,t)\,g(t)\,\mathrm dt,
			\qquad \lambda\ge 2,
		\end{equation}
		where $\psi\in C_c^\infty(\R^2)$ and $\mathcal S:\R^2\to\R$.
		Assume that
		\begin{enumerate}[leftmargin=2.2em]
			 \item (Support condition) 
			$\supp ψ(x,\cdot)\subset I_x,\quad|I_x|\le\delta$.
            
			\item (Derivative bounds in $t$) There exists $C\ge 1$ such that
			\begin{equation}\label{eq:PS_amp_bounds}
				|\partial_t^N\psi(x,t)|\le C\,\delta^{-N}
				\qquad \text{for all }(x,t)\in\R^2\text{ and }N=0,1,2.
			\end{equation}
			
			\item (Polynomial phase) For every fixed $x$, the map $t\mapsto \mathcal S(x,t)$ is a real polynomial.
			
			\item (Non-degeneracy) There exist constants $\mu>0$ and $A\ge 1$ such that
			\begin{equation*}
				0<\mu \le |\partial_{xt}^2 \mathcal S(x,t)|\le A\mu
				\qquad \text{for all }(x,t)\in\supp\psi.
			\end{equation*}
		\end{enumerate}
		Then $\mathcal{T}_\lambda:L^2(\R)\to L^2(\R)$ is bounded and
		\begin{equation}\label{eq:PS_bound}
			\|\mathcal{T}_\lambda\|_{L^2(\R)\to L^2(\R)}
			\le C'(\lambda\mu)^{-1/2},
		\end{equation}
		where the constant $C'$ depends only on
        $A$, the bounds in \eqref{eq:PS_amp_bounds},
	\end{lemma}

	\subsection{Dyadic localization in $(x,y,t)$}\label{subsec:dyadic}
	
	We now localize the operator \eqref{eq:intro_operator} dyadically in each variable.
	Fix $\varphi\in C_c^\infty(\R)$ such that
	\begin{equation*}
		\supp\varphi \subset \Bigl\{s\in\R:\frac12\le |s|\le 2\Bigr\},
	\end{equation*}
	and
	\begin{equation*}
		\sum_{j\in\Z}\varphi(2^{-j}s)=1\qquad\text{for all }s\neq 0.
	\end{equation*}
	For each $j\in\Z$, set
	\[
	\varphi_j(s):=\varphi(2^{-j}s).
	\]
	Since $\supp a$ is compact,
    only finitely many positive dyadic indices occur.
    Absorbing these into the implicit constants,
    we henceforth restrict to $\ell,m,n\le0$.

	Define the localized amplitudes
	\begin{equation*}
		a_{\ell,m,n}(x,y,t)
		=
		a(x,y,t)\,\varphi_\ell(x)\,\varphi_m(y)\,\varphi_n(t),
		\qquad \ell,m,n\le0,
	\end{equation*}
	and the corresponding localized operators
	\begin{equation}\label{eq:Tlmn_def}
		T_\lambda^{\ell,m,n}f(x,y)
		=
		\int_{\R}\e^{\,i\lambda\Phi(x,y,t)}\,a_{\ell,m,n}(x,y,t)\,f(t)\,\mathrm dt.
	\end{equation}
	Since the coordinate hyperplanes have measure zero,
    the kernel of $T_\lambda$ agrees almost everywhere with the sum of the localized kernels. Hence
    \begin{equation*}
        T_\lambda=\sum_{\ell,m,n\le0}T^{\ell,m,n}_\lambda .
    \end{equation*}
	On \(\supp a_{\ell,m,n}\) one has the dyadic scale relations
	\begin{equation*}
		|x|\sim 2^\ell,\qquad |y|\sim 2^m,\qquad |t|\sim 2^n.
	\end{equation*}
	Moreover,
	\begin{equation*}
		|\supp a_{\ell,m,n}|
		\lesssim
		2^{\ell+m+n}.
	\end{equation*}
	For \(N=0,1,2\), the \(t\)-derivative bounds
	\begin{equation*}
		|\partial_t^N a_{\ell,m,n}(x,y,t)|
		\lesssim
		2^{-nN}
	\end{equation*}
	hold uniformly in \(\ell,m,n\le0\). Indeed, each \(t\)-derivative falling on \(\varphi_n(t)=\varphi(2^{-n}t)\) produces a factor \(2^{-n}\), while derivatives falling on \(a\) are uniformly bounded; since \(n\le0\), these terms are also bounded by \(C_N2^{-nN}\).

	In the next section, we obtain $L^2$ estimates for each localized operator \(T_\lambda^{\ell,m,n}\), optimize among the available dyadic bounds, and reduce the proof of Theorem~\ref{thm:main} to a discrete triple-sum estimate.

	\section{Proof of Theorem~\ref{thm:main}: reduction to a triple sum}\label{sec:thm_main}
	
	In this section we prove the upper bound in Theorem~\ref{thm:main} assuming
	Proposition~\ref{prop:triplesum}. The proof proceeds in two steps:
	\begin{enumerate}[label=\textup{(\arabic*)},leftmargin=2.2em]
		\item establish localized estimates for each dyadic piece $T_\lambda^{\ell,m,n}$;
		\item sum these bounds over $(\ell,m,n)\in\Z_{\le0}^3$ after optimizing at each scale.
	\end{enumerate}
	Throughout, $\ell,m,n\le 0$ and $T_\lambda^{\ell,m,n}$ is defined in \eqref{eq:Tlmn_def}.
	
	We first obtain a size estimate for $T_\lambda^{\ell,m,n}$.
	
	\begin{lemma}\label{lem:size}
		For each $\ell,m,n\le 0$,
		\begin{equation*}
			\|T_{\lambda}^{\ell,m,n}\|_{L^2(\R)\to L^2(\R^2)}
			\lesssim 2^{(\ell+m+n)/2}.
		\end{equation*}
		The implicit constant depends only on finitely many seminorms of $a$ and the fixed cutoff $\varphi$.
	\end{lemma}
	
	\begin{proof}
		Write $T_{\lambda}^{\ell,m,n}$ as an integral operator from $L^2(\R_t)$ to $L^2(\R^2_{x,y})$ with kernel
		\[
		K_{\ell,m,n}(x,y;t)
		=
		\e^{\,i\lambda\Phi(x,y,t)}\,a_{\ell,m,n}(x,y,t).
		\]
		By the Hilbert--Schmidt inequality,
		\begin{equation}\label{eq:HS}
			\|T_{\lambda}^{\ell,m,n}\|_{2\to2}
			\le
			\|K_{\ell,m,n}\|_{L^2(\R^2_{x,y}\times\R_t)}
			=
			\Bigl(\iiint_{\R^3}|a_{\ell,m,n}(x,y,t)|^2\,\mathrm dx\,\mathrm dy\,\mathrm dt\Bigr)^{1/2}.
		\end{equation}
			Since both $a$ and $\Phi$ are uniformly bounded, we have
			\[
			|a_{\ell,m,n}(x,y,t)|\le C_a.
			\]
			Moreover,
			\[
			\supp a_{\ell,m,n}
			\subset
			\supp a\cap
			\{ |x|\sim 2^\ell,\ |y|\sim 2^m,\ |t|\sim 2^n\}.
			\]
			Hence
			\[
			|\supp a_{\ell,m,n}|
			\lesssim
			2^{\ell+m+n}.
			\]
			Substituting this support estimate into \eqref{eq:HS} gives
			\[
			\|T_{\lambda}^{\ell,m,n}\|_{2\to2}
			\lesssim
			2^{(\ell+m+n)/2}.
			\]
	\end{proof}
	
	We next exploit oscillation using the $(1+1)$-dimensional estimate Lemma~\ref{lem:PS}.
	The relevant mixed derivatives of the phase are
	\begin{equation*}
		\partial_{xt}^2\Phi(x,y,t)
		=
		(k-k_P)k_P\,x^{k-k_P-1}t^{k_P-1},
		\quad
		\partial_{yt}^2\Phi(x,y,t)
		=
		(k-k_Q)k_Q\,y^{k-k_Q-1}t^{k_Q-1}.
	\end{equation*}
	On $\supp a_{\ell,m,n}$, one has
	\begin{equation}\label{eq:mu}
		|\partial_{xt}^2\Phi|
		\approx
		2^{(k-k_P-1)\ell}\,2^{(k_P-1)n},
		\qquad
		|\partial_{yt}^2\Phi|
		\approx
		2^{(k-k_Q-1)m}\,2^{(k_Q-1)n}.
	\end{equation}
		Then by Lemma~\ref{lem:PS}, we have
	\begin{lemma}\label{lem:osc}
		For each $\ell,m,n\le 0$,
		\begin{align}\label{eq:osc}
			\begin{split}
				&\|T_{\lambda}^{\ell,m,n}\|_{L^2(\R)\to L^2(\R^2)}
				\\
				\lesssim\;
				&\min\Bigl\{
				(\lambda\,2^{(k-k_P-1)\ell}2^{(k_P-1)n})^{-1/2}\,2^{m/2},\;
				(\lambda\,2^{(k-k_Q-1)m}2^{(k_Q-1)n})^{-1/2}\,2^{\ell/2}
				\Bigr\}.
			\end{split}
		\end{align}
	\end{lemma}
	
	\begin{proof}
		We prove the first bound in \eqref{eq:osc} for fixed $y\in \supp(\varphi_m)$.  The second bound is obtained in the same way by fixing $x$, applying Lemma~\ref{lem:PS} in the $(y,t)$ variables, and then integrating over the $x$-support, whose measure is $O(2^\ell)$.
		
		Fix $y\in\supp\varphi_m$ and define an operator from $L^2(\R_t)$ to $L^2(\R_x)$ by
		\begin{equation}\label{eq:T_y_def}
			(\mathcal{T}_{\lambda,y}^{\ell,m,n}f)(x)
			:=
			\int_{\R}\e^{\,i\lambda\Phi(x,y,t)}\,a_{\ell,m,n}(x,y,t)\,f(t)\,\mathrm dt.
		\end{equation}
		Set
		\[
		\psi_y(x,t):=a_{\ell,m,n}(x,y,t),
		\qquad
		\Phi_y(x,t):=\Phi(x,y,t).
		\]
		We verify the hypotheses of Lemma~\ref{lem:PS} for the operator \eqref{eq:PS_op} with phase $\Phi_y$ and amplitude $\psi_y$.
		
		\begin{itemize}[leftmargin=2.2em]
			\item \emph{Support and $t$-width.}
			Since $\psi_y(x,t)$ contains the factor $\varphi_n(t)$, its $t$-support lies in
			\[
			\{t:2^{n-1}\le |t|\le 2^{n+1}\}.
			\]
			After the finite sign-splitting described above, each piece is supported in a rectangle whose width in the $t$-direction is
			\[
			\delta\approx 2^n.
			\]
			
			\item \emph{$t$-derivative bounds.}
			For $N=0,1,2$,
			\[
			|\partial_t^N\psi_y(x,t)|
			=
			|\partial_t^N a_{\ell,m,n}(x,y,t)|
			\lesssim 2^{-nN}
			\approx \delta^{-N}.
			\]
				Indeed, every derivative falling on $\varphi_n(t)=\varphi(2^{-n}t)$ produces a factor $2^{-n}$, while derivatives falling on $a$ are uniformly bounded.  Since $n\le0$, the latter contribution is also bounded by $C_N2^{-nN}$.

			\item \emph{Polynomial phase.}
			For each fixed $x$ and $y$, the map
			\[
			t\mapsto \Phi_y(x,t)=\Phi(x,y,t)
			\]
			is a real polynomial.
			
			\item \emph{Non-degeneracy.}
			By \eqref{eq:mu}, on $\supp\psi_y$ one has
			\[
			|\partial_{xt}^2 \Phi_y(x,t)|
			=
			|\partial_{xt}^2\Phi(x,y,t)|
			\approx
			2^{(k-k_P-1)\ell}\,2^{(k_P-1)n}.
			\]
		\end{itemize}
		Thus Lemma~\ref{lem:PS} applies with
		\[
		\mu
		\approx
		2^{(k-k_P-1)\ell}\,2^{(k_P-1)n},
		\]
		and with constants independent of $y,\ell,m,n,\lambda$.  It follows that
		\begin{equation}\label{eq:fixed_y_bound}
			\|\mathcal{T}_{\lambda,y}^{\ell,m,n}\|_{L^2(\R_t)\to L^2(\R_x)}
			\lesssim
			(\lambda\,2^{(k-k_P-1)\ell}2^{(k_P-1)n})^{-1/2},
			\qquad y\in\supp\varphi_m.
		\end{equation}
		
		By Fubini's theorem and \eqref{eq:T_y_def},
		\[
		\|T_\lambda^{\ell,m,n}f\|_{L^2(\R^2_{x,y})}^2
		=
		\int_{\supp(\varphi_m)}
		\|\mathcal{T}_{\lambda,y}^{\ell,m,n}f\|_{L^2(\R_x)}^2\,\mathrm dy.
		\]
		Using \eqref{eq:fixed_y_bound} and
		\[
		|\supp\varphi_m|\lesssim 2^m,
		\]
		we obtain
		\begin{align*}
			\|T_\lambda^{\ell,m,n}f\|_{L^2(\R^2)}^2
			&\lesssim
			\int_{\supp\varphi_m}
			(\lambda\,2^{(k-k_P-1)\ell}2^{(k_P-1)n})^{-1}\,
			\|f\|_2^2\,\mathrm dy
			\\
			&\lesssim
			(\lambda\,2^{(k-k_P-1)\ell}2^{(k_P-1)n})^{-1}\,
			2^m\,\|f\|_2^2.
		\end{align*}
		Taking square roots yields the first term in \eqref{eq:osc}.

			For the second term, we change the roles of $x$ and $y$. Then fix $x\in\supp\varphi_\ell$ and define
			\[
			(\mathcal{T}_{\lambda,x}^{\ell,m,n}f)(y)
			:=
			\int_{\R}
			\e^{\,i\lambda\Phi(x,y,t)}a_{\ell,m,n}(x,y,t)f(t)\,\mathrm dt.
			\]
			Apply Lemma~\ref{lem:PS} in the variables $(y,t)$.  On $\supp a_{\ell,m,n}$,
			\[
            |\partial_{yt}^2\Phi_x(y,t)|
            =
			|\partial_{yt}^2\Phi(x,y,t)|
			\approx
			2^{(k-k_Q-1)m}2^{(k_Q-1)n},
			\]
			and the same $t$-width and derivative bounds hold.  Therefore
			\[
			\|\mathcal{T}_{\lambda,x}^{\ell,m,n}\|_{L^2_t\to L^2_y}
			\lesssim
			(\lambda\,2^{(k-k_Q-1)m}2^{(k_Q-1)n})^{-1/2}.
			\]
			Integrating the square of this bound over the $x$-support, whose measure is $O(2^\ell)$, gives
			\[
			\|T_{\lambda}^{\ell,m,n}\|_{2\to2}
			\lesssim
			(\lambda\,2^{(k-k_Q-1)m}2^{(k_Q-1)n})^{-1/2}2^{\ell/2}.
			\]
		This proves \eqref{eq:osc}.
	\end{proof}
	
	\subsection{The triple sum}
	
	Define, for $\ell,m,n\le 0$,
	\begin{align}
		A_{\ell,m,n}&:=2^{(\ell+m+n)/2},\label{eq:defA}\\
		B_{\ell,m,n}&:=(\lambda\,2^{(k-k_P-1)\ell}2^{(k_P-1)n})^{-1/2}\,2^{m/2},\label{eq:defB}\\
		C_{\ell,m,n}&:=(\lambda\,2^{(k-k_Q-1)m}2^{(k_Q-1)n})^{-1/2}\,2^{\ell/2}.\label{eq:defC}
	\end{align}
	Lemma~\ref{lem:size} gives $\|T_\lambda^{\ell,m,n}\|_{2\to2}\lesssim A_{\ell,m,n}$, while
	Lemma~\ref{lem:osc} gives 
    \[
        \|T_\lambda^{\ell,m,n}\|_{2\to2}\lesssim \min\{B_{\ell,m,n},C_{\ell,m,n}\}.
    \]
	Combining these, one has
	\begin{equation*}
		\|T_\lambda^{\ell,m,n}\|_{2\to2}
		\lesssim
		\min\{A_{\ell,m,n},B_{\ell,m,n},C_{\ell,m,n}\}.
	\end{equation*}
	
	Define the full nonnegative triple sum
	\begin{equation}\label{eq:triplesum_def}
		S(\lambda)
		:=
		\sum_{\ell,m,n\le0}
		\min\{A_{\ell,m,n},B_{\ell,m,n},C_{\ell,m,n}\}.
	\end{equation}
	The preceding estimates show that every localized operator is controlled by
    the minimum of three quantities:
    a trivial size estimate and two oscillatory estimates.
    Consequently, the proof of Theorem~1.2 reduces to estimating the resulting triple sum.
    \begin{proposition}\label{prop:triplesum}
			Let \(S(\lambda)\) be defined by \eqref{eq:triplesum_def}.  Then, for all \(\lambda\ge2\),
			\begin{equation}\label{eq:triplesum_bound}
				S(\lambda)
				\lesssim
				\lambda^{-\rho_*}(\log\lambda)^{\gamma},
			\end{equation}
			where \(\rho_*\) is defined in \eqref{eq:rho_star} and \(\gamma\) is defined in \eqref{eq:gamma_def}.
	\end{proposition}
	
	By Proposition~\ref{prop:triplesum}, we obtain the desired upper bound
	\[
	\|T_\lambda\|_{2\to2}
	\lesssim
	\lambda^{-\rho_*}(\log\lambda)^\gamma,
	\]
	which is \eqref{eq:upperbound_main}. Hence, it suffices to prove Proposition~\ref{prop:triplesum}.

		For later use in Section~\ref{sec:triplesum}, we record the elementary algebraic comparisons between
		\(A_{\ell,m,n}\), \(B_{\ell,m,n}\), and \(C_{\ell,m,n}\).  Writing \(L=\log_2\lambda\), one has
		\[
		A_{\ell,m,n}\le B_{\ell,m,n}
		\quad\Longleftrightarrow\quad
		(k-k_P)\ell+k_P n\le -L,
		\]
		\[
		A_{\ell,m,n}\le C_{\ell,m,n}
		\quad\Longleftrightarrow\quad
		(k-k_Q)m+k_Q n\le -L,
		\]
		and
		\[
		B_{\ell,m,n}\le C_{\ell,m,n}
		\quad\Longleftrightarrow\quad
		(k-k_Q)m+(k_Q-k_P)n\le (k-k_P)\ell.
		\]
		These equivalences follow directly by taking logarithms base \(2\) in the definitions
		\eqref{eq:defA}--\eqref{eq:defC}.

	The proof of Proposition~\ref{prop:triplesum} is a discrete summation argument over $(\ell,m,n)\in\Z_{\le0}^3$, obtained by partitioning the index set into regions where one of
	$A_{\ell,m,n},B_{\ell,m,n},C_{\ell,m,n}$ realizes the minimum. Since this computation is combinatorial and lengthy, we defer it to Section~\ref{sec:triplesum}.

	\section{Proof of Theorem~\ref{thm:lbound}}\label{sec:lower}
	This section establishes the lower bound in Theorem~\ref{thm:lbound} by a scale-adapted Knapp-type argument.
    We take a test function supported on a short interval in the $t$-variable,
    while $(x,y)$ is restricted to a region where the phase remains essentially
    constant.
    Consequently, the oscillatory integral behaves like the length of the
    integration interval.
	
	\begin{proof}[Proof of Theorem~\ref{thm:lbound}]
		Since the estimate is invariant under replacing $a$ by $-a$,
        we may assume
        \[
        a(0,0,0)>0.
        \]
		By continuity, there exist constants $r_0\in(0,1]$ and $c_0\in(0,1]$ such that
		$a(x,y,t)\neq0$ on $\{|x|,|y|,|t|\le r_0\}$ and
		\begin{equation}\label{eq:amp_sector}
			|a(x,y,t)|\ge c_0
			\qquad\text{whenever }|x|,|y|,|t|\le r_0.
		\end{equation}
		
		Fix a constant $r_1\in(0,r_0]$ to be specified below and let $\delta\ge 0$.
		Define
		\begin{equation*}
			f_\lambda(t):=\mathbf{1}_{[-r_1\lambda^{-\delta},\,r_1\lambda^{-\delta}]}(t).
		\end{equation*}
		Then
		\begin{equation}\label{eq:test_L2_norm}
			\|f_\lambda\|_{L^2(\R)}=(2r_1)^{1/2}\lambda^{-\delta/2}.
		\end{equation}
		Let $c\in(0,r_0]$ be a small constant to be chosen below and define, for $\lambda\ge 2$,
		\begin{equation}\label{eq:Omega_def}
			\Omega_\lambda(\delta)
			:=\Bigl\{(x,y)\in\R^2:\ |x|\le c\,X_\lambda(\delta),\ |y|\le c\,Y_\lambda(\delta)\Bigr\},
		\end{equation}
		where
		\begin{equation}\label{eq:XY_def}
			X_\lambda(\delta):=\min\{\lambda^{(k_P\delta-1)/(k-k_P)},\,1\},\qquad
			Y_\lambda(\delta):=\min\{\lambda^{(k_Q\delta-1)/(k-k_Q)},\,1\}.
		\end{equation}
		By construction, $|x|,|y|\le c\le r_0$ for all $(x,y)\in\Omega_\lambda(\delta)$.
		Moreover, since $r_1\le r_0$ and $\delta\ge0$, we also have $|t|\le r_0$ for all
		$t\in\supp f_\lambda$ and all $\lambda\ge2$.
		Therefore, for $(x,y)\in\Omega_\lambda(\delta)$ and $t\in\supp f_\lambda$ we may invoke
		\eqref{eq:amp_sector}.
		The key point is the following estimate:
		\begin{align}\label{T:lbound}
			|T_\lambda f_\lambda(x,y)| \gtrsim \lambda^{-\delta} \mathbf{1}_{\Omega_\lambda(\delta)}(x,y).
		\end{align}

		Indeed, for $(x,y)\in\Omega_\lambda(\delta)$ and $|t|\le r_1\lambda^{-\delta}$, we have
		\begin{align*}
			\lambda|x^{k-k_P}t^{k_P}|
			&\le \lambda\,(cX_\lambda(\delta))^{k-k_P}\,(r_1\lambda^{-\delta})^{k_P}.
		\end{align*}
		If $\delta\le 1/k_P$, then
		\[
		X_\lambda(\delta)=\lambda^{(k_P\delta-1)/(k-k_P)},
		\]
		hence
		\[
		\lambda\,(cX_\lambda(\delta))^{k-k_P}\,(r_1\lambda^{-\delta})^{k_P}
		=\lambda\,c^{k-k_P}\lambda^{k_P\delta-1}\,r_1^{k_P}\lambda^{-k_P\delta}
		=c^{k-k_P}r_1^{k_P}.
		\]
		If $\delta\ge 1/k_P$, then $X_\lambda(\delta)=1$, hence
		\[
		\lambda\,(cX_\lambda(\delta))^{k-k_P}\,(r_1\lambda^{-\delta})^{k_P}
		\le \lambda\,c^{k-k_P}\,r_1^{k_P}\lambda^{-k_P\delta}
		=c^{k-k_P}r_1^{k_P}\lambda^{1-k_P\delta}
		\le c^{k-k_P}r_1^{k_P}.
		\]
		Thus, in all cases,
		\begin{equation}\label{eq:phase_term1}
			\lambda|x^{k-k_P}t^{k_P}|\le c^{k-k_P}r_1^{k_P}.
		\end{equation}
		By the same argument, using $Y_\lambda(\delta)$ and $k_Q$ in place of $X_\lambda(\delta)$ and $k_P$,
		\begin{equation}\label{eq:phase_term2}
			\lambda|y^{k-k_Q}t^{k_Q}|\le c^{k-k_Q}r_1^{k_Q}.
		\end{equation}
		Combining \eqref{eq:phase_term1}--\eqref{eq:phase_term2} yields
		\[
		|\lambda\Phi(x,y,t)|
		\le c^{k-k_P}r_1^{k_P}+c^{k-k_Q}r_1^{k_Q}.
		\]
		Choose $r_1\in(0,r_0]$ and then $c\in(0,r_0]$ so that
		\[
		c^{k-k_P}r_1^{k_P}+c^{k-k_Q}r_1^{k_Q}\le \frac{\pi}{20}.
		\]
		This proves \begin{equation*}
		 	|\lambda\Phi(x,y,t)|\le \frac{\pi}{20}
		 	\qquad\text{whenever }(x,y)\in\Omega_\lambda(\delta)\text{ and }t\in\supp f_\lambda.
		 \end{equation*}

		Together with \eqref{eq:amp_sector} we obtain
		\begin{equation*}
			\e^{\,i\lambda\Phi(x,y,t)}a(x,y,t)\bigr)
			\gtrsim c_0
			\qquad \text{for all }t\in\supp f_\lambda.
		\end{equation*}
		Since
        $f_\lambda\ge0$
        and
        $\Re(e^{i\lambda\Phi}a)\gtrsim c_0$
        on the support,, we conclude that for all
		$(x,y)\in\Omega_\lambda(\delta)$,
		\begin{align}
			|T_\lambda f_\lambda(x,y)|
			&\ge
			 \int_{\supp (f_\lambda)}|a(x,y,t)|\,\mathrm dt\nonumber\\
			&\gtrsim c_0\,|\supp (f_\lambda)|\nonumber\\
			&= c_0\,(2r_1)\lambda^{-\delta}
			\ \gtrsim\ \lambda^{-\delta}.\label{eq:pointwise_lower}
		\end{align}
		By \eqref{eq:pointwise_lower}, we obtain \eqref{T:lbound}.
		
		Integrating \eqref{T:lbound} over $x,y$ yields
		\begin{equation}\label{eq:L2_lower_raw}
			\|T_\lambda f_\lambda\|_{L^2(\R^2)}
			\ge \|T_\lambda f_\lambda\|_{L^2(\Omega_\lambda(\delta))}
			\gtrsim \lambda^{-\delta}\,|\Omega_\lambda(\delta)|^{1/2}.
		\end{equation}
		Combining \eqref{eq:L2_lower_raw} with \eqref{eq:test_L2_norm} gives
		\begin{equation}\label{eq:op_lower_ratio}
			\|T_\lambda\|_{2\to2}
			\ge \frac{\|T_\lambda f_\lambda\|_2}{\|f_\lambda\|_2}
			\gtrsim \lambda^{-\delta/2}\,|\Omega_\lambda(\delta)|^{1/2}.
		\end{equation}
		By \eqref{eq:Omega_def}--\eqref{eq:XY_def}, we have
		\begin{equation}\label{eq:Omega_measure}
			|\Omega_\lambda(\delta)| \approx X_\lambda(\delta)\,Y_\lambda(\delta),
		\end{equation}
		with constants depending only on $c$.
		Thus, the lower bound in \eqref{eq:op_lower_ratio} is determined by the size of
		$X_\lambda(\delta)$ and $Y_\lambda(\delta)$, leading to three regimes.
		
		\smallskip
		\noindent\textbf{Case 1: $0\le \delta\le 1/k_Q$.}
		Then
		\[
		X_\lambda(\delta)=\lambda^{(k_P\delta-1)/(k-k_P)},
		\qquad
		Y_\lambda(\delta)=\lambda^{(k_Q\delta-1)/(k-k_Q)}.
		\]
		By \eqref{eq:op_lower_ratio}--\eqref{eq:Omega_measure},
		\begin{align}
			\|T_\lambda\|_{2\to2}
			&\gtrsim \lambda^{-\delta/2}\,
			\lambda^{\frac12\bigl(\frac{k_P\delta-1}{k-k_P}+\frac{k_Q\delta-1}{k-k_Q}\bigr)}\nonumber\\
			&=\lambda^{-\rho_1-\frac{\delta}{2}\bigl(1-\frac{k_P}{k-k_P}-\frac{k_Q}{k-k_Q}\bigr)}.\label{eq:case1_bound}
		\end{align}
		The exponent in \eqref{eq:case1_bound} is affine in $\delta$. Hence the optimum is attained at one of the endpoints.
		If
		\[
		1-\frac{k_P}{k-k_P}-\frac{k_Q}{k-k_Q}\ge 0,
		\]
		then the lower bound in \eqref{eq:case1_bound} is optimized at $\delta=0$, giving
		\[
		\|T_\lambda\|_{2\to2}\gtrsim \lambda^{-\rho_1}.
		\]
		If
		\[
		1-\frac{k_P}{k-k_P}-\frac{k_Q}{k-k_Q}<0,
		\]
		then the lower bound in \eqref{eq:case1_bound} is optimized at $\delta=1/k_Q$, giving
		\[
		\|T_\lambda\|_{2\to2}\gtrsim \lambda^{-\rho_2}.
		\]
		
		\smallskip
		\noindent\textbf{Case 2: $1/k_Q\le \delta\le 1/k_P$.}
		Then $Y_\lambda(\delta)=1$ while
		\[
		X_\lambda(\delta)=\lambda^{(k_P\delta-1)/(k-k_P)}.
		\]
		Hence
		\begin{align}
			\|T_\lambda\|_{2\to2}
			&\gtrsim \lambda^{-\delta/2}\,\lambda^{\frac12\cdot\frac{k_P\delta-1}{k-k_P}}
			=\lambda^{-\frac1{2(k-k_P)}-\frac{\delta}{2}\bigl(1-\frac{k_P}{k-k_P}\bigr)}.\label{eq:case2_bound}
		\end{align}
		Since
		\[
		1-\frac{k_P}{k-k_P}=\frac{k-2k_P}{k-k_P},
		\]
		the exponent in \eqref{eq:case2_bound} is affine in $\delta$. Again,
        the exponent is affine in $\delta$,
        so it suffices to compare the two endpoints:
		\begin{itemize}
			\item if $k>2k_P$, the lower bound is optimized at $\delta=1/k_Q$, yielding
			\[
			\|T_\lambda\|_{2\to2}\gtrsim \lambda^{-\rho_2};
			\]
			\item if $k\le 2k_P$, the lower bound is optimized at $\delta=1/k_P$, yielding
			\[
			\|T_\lambda\|_{2\to2}\gtrsim \lambda^{-\rho_3}.
			\]
		\end{itemize}
		
		\smallskip
		\noindent\textbf{Case 3: $\delta\ge 1/k_P$.}
		Then $X_\lambda(\delta)=Y_\lambda(\delta)=1$.
			Taking the endpoint choice $\delta=1/k_P$ in \eqref{eq:op_lower_ratio}, we obtain
			\[
			\|T_\lambda\|_{2\to2}
			\gtrsim
			\lambda^{-1/(2k_P)}
			=
			\lambda^{-\rho_3}.
			\]
			For $\delta>1/k_P$, the lower bound $\lambda^{-\delta/2}$ is weaker, so the optimal choice in this regime is
			$\delta=1/k_P$.

			Collecting the three regimes,
we obtain the lower bounds at
			\[
			\delta=0,\qquad \delta=\frac1{k_Q},\qquad \delta=\frac1{k_P}
			\]
			are admissible choices in the three regimes above and produce, respectively, the lower bounds
			\[
			\lambda^{-\rho_1},\qquad
			\lambda^{-\rho_2},\qquad
			\lambda^{-\rho_3}.
			\]
			Taking the strongest among them yields, for all sufficiently large $\lambda$,
		\[
		\|T_\lambda\|_{2\to2}
		\gtrsim \lambda^{-\min\{\rho_1,\rho_2,\rho_3\}}
		=\lambda^{-\rho_*},
		\]
		which is \eqref{eq:lowerbound_main}.
	\end{proof}

	\section{Proof of Proposition~\ref{prop:triplesum}}\label{sec:triplesum}
    Recall the dyadic quantities $A_{\ell,m,n}$, $B_{\ell,m,n}$, and
    $C_{\ell,m,n}$ defined in \eqref{eq:defA}--\eqref{eq:defC}, and let
    \[
    S(\lambda)
    =\sum_{\ell,m,n\le0}
    \min\{A_{\ell,m,n},B_{\ell,m,n},C_{\ell,m,n}\}.
    \]

    Write
    \[
    L=\log_2\lambda .
    \]
    
    The inequalities at the end of Section~3 partition the index set
    $\mathbb Z_{\le0}^3$ into three regions according to which of
    $A_{\ell,m,n}$,
    $B_{\ell,m,n}$,
    and
    $C_{\ell,m,n}$
    realizes the minimum.
    Accordingly, we estimate the contribution from each region separately.
    
    Throughout this section we repeatedly use the elementary geometric-sum
    principle
    \[
    \sum_{j=0}^{N}2^{\alpha j}
    \lesssim
    \begin{cases}
    2^{\alpha N},&\alpha>0,\\
    N+1,&\alpha=0,\\
    1,&\alpha<0,
    \end{cases}
    \]
    together with
    \[
    \sum_{j\ge N}2^{-\alpha j}
    \lesssim
    2^{-\alpha N},
    \qquad
    \alpha>0.
    \]

    To simplify the notation, we freely replace real-valued endpoints of the
    summation ranges by their integer parts.
    Since all estimates are up to absolute multiplicative constants, this affects
    the estimates only by an absolute factor.

	\subsection{Region I: $A\le B$ and $A\le C$}

    In this region the trivial size estimate dominates both oscillatory
    estimates.
    After summing over the spatial variables, the remaining behavior is governed
    by the $t$-scale.
    Depending on whether one or both affine constraints become inactive, three
    different ranges of the dyadic index $n$ arise naturally.
    
	The conditions $A_{\ell,m,n}\le B_{\ell,m,n}$ and $A_{\ell,m,n}\le C_{\ell,m,n}$ are equivalent to
	\begin{equation}\label{eq:AB_app}
		(k-k_P)\ell+k_P n\le -L,
	\end{equation}
	and
	\begin{equation}\label{eq:AC_app}
		(k-k_Q)m+k_Q n\le -L,
	\end{equation}
	respectively. Hence, the contribution of Region~I is
	\[
	S_A:=\sum_{\substack{\ell,m,n\le0\\ \eqref{eq:AB_app},\,\eqref{eq:AC_app}}}A_{\ell,m,n}
	=\sum_{\substack{\ell,m,n\le0\\ \eqref{eq:AB_app},\,\eqref{eq:AC_app}}}2^{(\ell+m+n)/2}.
	\]
	Fix $n\le 0$. The constraints \eqref{eq:AB_app}, \eqref{eq:AC_app} imply
	\[
	\ell\le -\max\Bigl\{\frac{L+k_P n}{k-k_P},0\Bigr\},\qquad
	m\le -\max\Bigl\{\frac{L+k_Q n}{k-k_Q},0\Bigr\}.
	\]
	Summing the geometric series in $\ell$ and $m$ gives
	\[
	S_A\;\lesssim\;\sum_{n\le 0}2^{n/2}\,
	2^{-\frac12\max\{\frac{L+k_P n}{k-k_P},0\}}\,
	2^{-\frac12\max\{\frac{L+k_Q n}{k-k_Q},0\}}.
	\]
    The behavior changes according to the sign of
$L+k_Qn$
and
$L+k_Pn$.
Hence we split the summation into three ranges.

	\smallskip
	\noindent\emph{(i) The range $-\frac{L}{k_Q}\le n\le 0$.}
	Here $L+k_P n\ge 0$ and $L+k_Q n\ge 0$, so
	\begin{align*}
		S_A\lesssim S_{A,1}
		&:=\sum_{-L/k_Q\le n\le 0}
		2^{n/2}\,2^{-(L+k_P n)/(2(k-k_P))}\,2^{-(L+k_Q n)/(2(k-k_Q))}\\
		&=\lambda^{-\rho_1}\sum_{-L/k_Q\le n\le 0}
		2^{\frac12\bigl(1-\frac{k_P}{k-k_P}-\frac{k_Q}{k-k_Q}\bigr)\,n},
	\end{align*}
	where $\rho_1=\frac1{2(k-k_P)}+\frac1{2(k-k_Q)}$.
	If
	\[
	1-\frac{k_P}{k-k_P}-\frac{k_Q}{k-k_Q}>0,
	\]
	then the summand is dominated by the endpoint \(n=0\), and the sum is $\lesssim 1$.
	If equality holds, the summand is constant and the sum has $\sim L$ terms, i.e.\ $\lesssim L$.
	If the coefficient is negative, the sum is dominated by the lower endpoint $n=-L/k_Q$, yielding
	\[
	S_{A,1}\lesssim \lambda^{-\rho_1}\,
	2^{-\frac{L}{2k_Q}\bigl(1-\frac{k_P}{k-k_P}-\frac{k_Q}{k-k_Q}\bigr)}
	=\lambda^{-\rho_2},
	\]
	since
	\[
	\rho_1-\frac1{2k_Q}\Bigl(\frac{k_P}{k-k_P}+\frac{k_Q}{k-k_Q}-1\Bigr)
	=
	\frac1{2(k-k_P)}+\frac1{2k_Q}\Bigl(1-\frac{k_P}{k-k_P}\Bigr)
	=\rho_2.
	\]
	Thus
	\begin{equation}\label{eq:SA1_bound}
		S_{A,1}\lesssim
		\lambda^{-\rho_1}(\log\lambda)^{\mathbf 1_{\eqref{eq:borderlineA}}}
		+\lambda^{-\rho_2}.
	\end{equation}
	
	\smallskip
	\noindent\emph{(ii) The range $-\frac{L}{k_P}\le n\le -\frac{L}{k_Q}$.}
	Here $L+k_P n\ge 0$ but $L+k_Q n\le 0$, so the $m$-constraint becomes $m\le 0$ and
	\begin{align*}
		S_A\lesssim S_{A,2}
		&:=\sum_{-L/k_P\le n\le -L/k_Q}
		2^{n/2}\,2^{-(L+k_P n)/(2(k-k_P))}\\
		&=\lambda^{-\frac1{2(k-k_P)}}\sum_{-L/k_P\le n\le -L/k_Q}
		2^{\frac12(1-\frac{k_P}{k-k_P})\,n}.
	\end{align*}
	Set
	\[
	\alpha:=\frac12\Bigl(1-\frac{k_P}{k-k_P}\Bigr)=\frac{k-2k_P}{2(k-k_P)}.
	\]
	If $k>2k_P$, then $\alpha>0$ and the sum is dominated by the upper endpoint $n=-L/k_Q$, giving
	\[
	S_{A,2}\lesssim \lambda^{-\rho_2}.
	\]
	If $k=2k_P$, then $\alpha=0$ and the sum has $\sim L$ terms, giving
	\[
	S_{A,2}\lesssim \lambda^{-1/(2(k-k_P))}L
	=\lambda^{-\rho_3}L,
	\]
	since then $\rho_3=\frac1{2k_P}=\frac1{2(k-k_P)}$.
	If $k<2k_P$, then $\alpha<0$ and the sum is dominated by $n=-L/k_P$, giving
	\[
	S_{A,2}\lesssim \lambda^{-1/(2k_P)}=\lambda^{-\rho_3}.
	\]
	Hence
	\begin{equation}\label{eq:SA2_bound}
		S_{A,2}\lesssim
		\begin{cases}
			\lambda^{-\rho_2},&k>2k_P,\\
			\lambda^{-\rho_3}\log\lambda,&k=2k_P,\\
			\lambda^{-\rho_3},&k<2k_P.
		\end{cases}
	\end{equation}
	
	\smallskip
	\noindent\emph{(iii) The range $n\le -\frac{L}{k_P}$.}
	Then $L+k_P n\le 0$ and $L+k_Q n\le 0$, so both $\ell,m$ are only constrained by $\ell,m\le 0$, and
	\begin{equation}\label{eq:SA3_bound}
		S_{A,3}:=\sum_{n\le -L/k_P}2^{n/2}\;\lesssim\;2^{-L/(2k_P)}=\lambda^{-\rho_3}.
	\end{equation}
	
	Combining \eqref{eq:SA1_bound}, \eqref{eq:SA2_bound}, and \eqref{eq:SA3_bound} yields
	\begin{equation}\label{eq:SA_final_rig}
		S_A\lesssim
		\begin{cases}
			\max\{\lambda^{-\rho_1}(\log\lambda)^{\mathbf 1_{\eqref{eq:borderlineA}}},\,\lambda^{-\rho_2},\,\lambda^{-\rho_3}\},&k\neq 2k_P,\\[2pt]
			\lambda^{-\rho_3}\log\lambda,&k=2k_P.
		\end{cases}
	\end{equation}
	
	\subsection{Region II: $B\le A$ and $B\le C$}
	
	Let
	\[
	S_B:=\sum_{\substack{\ell,m,n\le0\\ B\le A,\,B\le C}}B_{\ell,m,n}.
	\]
	The constraints $B\le A$ and $B\le C$ are equivalent to
	\begin{equation}\label{eq:BA_app}
		(k-k_P)\ell+k_P n\ge -L,
	\end{equation}
	and
	\begin{equation}\label{eq:BC_app}
		(k-k_Q)m+(k_Q-k_P)n\le (k-k_P)\ell,
	\end{equation}
	respectively.
    
	The argument differs substantially according to whether $k\le 2k_P$ or $k>2k_P$.
    Indeed, the exponent appearing in the $\ell$-summation changes sign exactly at
    $k=2k_P$,
    which determines whether the geometric series is dominated by its upper or lower endpoint.
    
	\subsubsection{The case $k\le 2k_P$}
	
	We drop the constraint \eqref{eq:BC_app}, which only enlarges the summation domain and therefore gives an upper bound, and estimate using \eqref{eq:BA_app} alone:
	\[
	S_B\le \sum_{m\le 0}\sum_{\substack{\ell,n\le 0\\ \eqref{eq:BA_app}}}B_{\ell,m,n}
	=
	\sum_{m\le 0}2^{-L/2}2^{m/2}
	\sum_{\substack{\ell,n\le 0\\ \eqref{eq:BA_app}}}
	2^{-\frac{k-k_P-1}{2}\ell}\,2^{-\frac{k_P-1}{2}n}.
	\]
	Fix $\ell$, and the condition \eqref{eq:BA_app} gives
	\[
	n\ge -\frac{L+(k-k_P)\ell}{k_P}.
	\]
	Since $n\le 0$, this forces $\ell\ge -\frac{L}{k-k_P}$. Thus $\ell\in[-L/(k-k_P),0]$, and for such $\ell$,
	\[
	\sum_{-\frac{L+(k-k_P)\ell}{k_P}\le n\le 0}
	2^{-\frac{k_P-1}{2}n}
	\;\lesssim\;
	2^{\frac{k_P-1}{2k_P}(L+(k-k_P)\ell)}.
	\]
		Here, there is no exceptional \(k_P=1\) case in the present subcase \(k\le 2k_P\), because \(1\le k_P<k_Q<k\) would then force \(k\le2\), which is impossible.
	Therefore
	\begin{align}\label{eq:SB}
    \begin{split}
		S_B
		&\lesssim
		\sum_{m\le 0}2^{-L/2}2^{m/2}
		\sum_{-L/(k-k_P)\le \ell\le 0}
		2^{-\frac{k-k_P-1}{2}\ell}\,
		2^{\frac{k_P-1}{2k_P}(L+(k-k_P)\ell)}\\
		&=
		\sum_{m\le 0}
		2^{-L/(2k_P)}2^{m/2}
		\sum_{-L/(k-k_P)\le \ell\le 0}
		2^{\bigl(1-\frac{k}{2k_P}\bigr)\ell}.
    \end{split}
	\end{align}
	If $k<2k_P$, then $1-\frac{k}{2k_P}>0$, and the $\ell$-sum is $\lesssim 1$.
	If $k=2k_P$, then the summand is constant, and the $\ell$-sum has $\sim L$ terms, hence $\lesssim L$.
	Since $\sum_{m\le 0}2^{m/2}\lesssim 1$, we conclude
	\begin{equation}\label{eq:SB_kle2kP}
		S_B\lesssim
		\begin{cases}
			\lambda^{-\rho_3},&k<2k_P,\\
			\lambda^{-\rho_3}\log\lambda,&k=2k_P.
		\end{cases}
	\end{equation}
	
	\subsubsection{The case $k>2k_P$}
	
	We split the $m$-sum into two ranges:
	\[
	m\le -\frac{L}{k-k_Q}
	\qquad\text{and}\qquad
	-\frac{L}{k-k_Q}\le m\le 0.
	\]
	
	Let $m\le -\frac{L}{k-k_Q}$.
	In this subrange the inequality \eqref{eq:BC_app} is redundant: indeed, since $n\le 0$,
	\[
	(k-k_Q)m+(k_Q-k_P)n\le (k-k_Q)m\le -L,
	\]
	while \eqref{eq:BA_app} implies
	\[
	(k-k_P)\ell\ge -L-k_P n\ge -L.
	\]
	Hence
	\[
	(k-k_Q)m+(k_Q-k_P)n\le (k-k_P)\ell,
	\]
	which is \eqref{eq:BC_app}. Therefore we may keep only \eqref{eq:BA_app}. We claim that
	\begin{equation}\label{eq:BA-sum-kgt}
		\sum_{\substack{\ell,n\le0\\ \eqref{eq:BA_app}}}
		2^{-L/2}
		2^{-\frac{k-k_P-1}{2}\ell}
		2^{-\frac{k_P-1}{2}n}
		\lesssim
		2^{-L/(2(k-k_P))}.
	\end{equation}
	For \(k_P>1\), this follows exactly from the preceding computation \eqref{eq:SB} without the $m$-term. Note that the $\ell$-sum dominated by
	\(\ell=-L/(k-k_P)\) because \(1-\frac{k}{2k_P}<0\).
		If \(k_P=1\), the inner \(n\)-sum is a length factor rather than a genuine geometric sum.  More precisely,
		\[
		\sum_{-L-(k-1)\ell\le n\le0}1
		\lesssim
		1+L+(k-1)\ell.
		\]
		Writing \(j=-\ell\) and using the summation by parts, the remaining sum is bounded by
		\[
		2^{-L/2}\sum_{0\le j\le L/(k-1)}
		\bigl(1+L-(k-1)j\bigr)2^{\frac{k-2}{2}j}
		\lesssim
		2^{-L/(2(k-1))},
		\]
		which is \eqref{eq:BA-sum-kgt} when \(k_P=1\).
	Therefore
	\begin{align*}
		\sum_{m\le -L/(k-k_Q)}
		\sum_{\substack{\ell,n\le 0\\ \eqref{eq:BA_app}}}
		B_{\ell,m,n}
		&\lesssim
		\Bigl(\sum_{m\le -L/(k-k_Q)}2^{m/2}\Bigr)
		2^{-L/(2(k-k_P))}\\
		&\lesssim
		2^{-L\bigl(\frac1{2(k-k_P)}+\frac1{2(k-k_Q)}\bigr)}
		=
		\lambda^{-\rho_1}.
	\end{align*}
	
	We now assume
	\[
	-\frac{L}{k-k_Q}\le m\le 0.
	\]
	Here \eqref{eq:BC_app} is effective. 
    For fixed $m$,
    the two lower bounds for $\ell$ intersect at a unique value of $n$.
    This transition point plays a crucial role in the summation and is therefore denoted by
	\begin{equation*}
		n_0(m):=-\frac{L+(k-k_Q)m}{k_Q}
		=-\frac{L}{k_Q}-\frac{k-k_Q}{k_Q}\,m,
	\end{equation*}
	which satisfies $-L/k_Q\le n_0(m)\le 0$ on this $m$-range. The two lower bounds for $\ell$ coming from
	\eqref{eq:BA_app} and \eqref{eq:BC_app} are
	\[
	\ell\ge \ell_1(n):=-\frac{L+k_P n}{k-k_P},
	\qquad
	\ell\ge \ell_2(m,n):=
	\frac{(k-k_Q)m+(k_Q-k_P)n}{k-k_P}.
	\]
	These coincide exactly when $n=n_0(m)$, and one checks that
	\[
	n\le n_0(m)\ \Longrightarrow\ \ell_1(n)\ge \ell_2(m,n),
	\qquad
	n\ge n_0(m)\ \Longrightarrow\ \ell_2(m,n)\ge \ell_1(n).
	\]
	Thus, we split the summation accordingly.
	Below the transition point, $\ell_1$ is the active constraint.

	\smallskip
	\noindent\emph{(b1) The range $-L/k_P\le n\le n_0(m)$.}
	Recall that $n\ge -L/k_P$ is necessary for \eqref{eq:BA_app} to admit any $\ell\le 0$.
	Here $\ell\ge \ell_1(n)$ is the active constraint, and since $B_{\ell,m,n}$ is increasing as $\ell$ decreases
	because of the factor $2^{-\frac{k-k_P-1}{2}\ell}$, the $\ell$-sum is dominated by $\ell=\ell_1(n)$:
	\[
	\sum_{\ell_1(n)\le \ell\le 0}
	2^{-\frac{k-k_P-1}{2}\ell}
	\lesssim
	2^{-\frac{k-k_P-1}{2}\ell_1(n)}
	=
	2^{\frac{k-k_P-1}{2(k-k_P)}(L+k_P n)}.
	\]
	Hence, for each fixed $m$,
	\begin{align*}
		&\sum_{-L/k_P\le n\le n_0(m)}
		\sum_{\substack{\ell\le 0\\ \eqref{eq:BA_app},\,\eqref{eq:BC_app}}}
		B_{\ell,m,n}\\
		\lesssim
		&\sum_{-L/k_P\le n\le n_0(m)}
		2^{-L/2}2^{m/2}\,
		2^{-\frac{k_P-1}{2}n}\,
		2^{\frac{k-k_P-1}{2(k-k_P)}(L+k_P n)}\\
		=&2^{-L/(2(k-k_P))}\,2^{m/2}
		\sum_{-L/k_P\le n\le n_0(m)}
		2^{\frac{k-2k_P}{2(k-k_P)}\,n}.
	\end{align*}
	Since $k>2k_P$, the exponent $\frac{k-2k_P}{2(k-k_P)}>0$, so the $n$-sum is dominated by $n=n_0(m)$, yielding
	\[
	\lesssim
	2^{-L/(2(k-k_P))}\,2^{m/2}\,
	2^{-\frac{k-2k_P}{2k_Q(k-k_P)}(L+(k-k_Q)m)}
	=
	\lambda^{-\rho_2}\,2^{\theta m},
	\]
	where
	\begin{equation}\label{eq:theta_def}
		\theta:=
		\frac12-\frac{k-k_Q}{2k_Q}\Bigl(1-\frac{k_P}{k-k_P}\Bigr)
		=
		\frac12-\frac{(k-2k_P)(k-k_Q)}{2k_Q(k-k_P)}.
	\end{equation}
	Thus,
	\begin{align}\label{eq:b1}
		\sum_{-L/k_P\le n\le n_0(m)}
		\sum_{\substack{\ell\le 0\\ \eqref{eq:BA_app},\,\eqref{eq:BC_app}}}
		B_{\ell,m,n}
		\lesssim
		\lambda^{-\rho_2}\,2^{\theta m}.
	\end{align}
	
    Above the transition point,
    $\ell_2$ becomes active.
    
	\smallskip
	\noindent\emph{(b2) The range $n_0(m)\le n\le 0$.}
	Here $\ell\ge \ell_2(m,n)$ is active, and again the $\ell$-sum is dominated by $\ell=\ell_2(m,n)$:
	\[
	\sum_{\ell\ge \ell_2(m,n),\,\ell\le 0}
	2^{-\frac{k-k_P-1}{2}\ell}
	\lesssim
	2^{-\frac{k-k_P-1}{2}\ell_2(m,n)}
	=
	2^{-\frac{k-k_P-1}{2(k-k_P)}
		\bigl((k-k_Q)m+(k_Q-k_P)n\bigr)}.
	\]
	Thus
	\begin{align*}
		&\sum_{n_0(m)\le n\le 0}
		\sum_{\substack{\ell\le 0\\ \eqref{eq:BA_app},\,\eqref{eq:BC_app}}}
		B_{\ell,m,n}\\
		\lesssim
		&\sum_{n_0(m)\le n\le 0}
		2^{-L/2}2^{m/2}\,
		2^{-\frac{k_P-1}{2}n}\,
		2^{-\frac{k-k_P-1}{2(k-k_P)}
			\bigl((k-k_Q)m+(k_Q-k_P)n\bigr)}\\
		=
		&2^{-L/2}
		2^{\bigl(\frac12-\frac{k-k_P-1}{2(k-k_P)}(k-k_Q)\bigr)m}
		\sum_{n_0(m)\le n\le 0}
		2^{-\beta n},
	\end{align*}
	where
	\begin{align}\label{eq:beta}
	\beta:=
	\frac{k_P-1}{2}
	+
	\frac{k-k_P-1}{2(k-k_P)}(k_Q-k_P)
	=
	\frac{k_Q-1}{2}
	-\frac{k_Q-k_P}{2(k-k_P)}.
	\end{align}
	Since $\beta>0$ and $n\le 0$, the $n$-sum is dominated by $n=n_0(m)$:
	\[
	\sum_{n_0(m)\le n\le 0}2^{-\beta n}
	\lesssim
	2^{-\beta n_0(m)}
	=
	2^{\frac{\beta}{k_Q}(L+(k-k_Q)m)}.
	\]
	Using the relation $-\frac12+\frac{\beta}{k_Q}=-\rho_2$ and combining the $m$-powers gives precisely the same factor $2^{\theta m}$ as in \eqref{eq:theta_def}. Hence
	\begin{align}\label{eq:b2}
		\sum_{n_0(m)\le n\le 0}
		\sum_{\substack{\ell\le 0\\ \eqref{eq:BA_app},\,\eqref{eq:BC_app}}}
		B_{\ell,m,n}
		\lesssim
		\lambda^{-\rho_2}\,2^{\theta m}.
	\end{align}
	
	\smallskip
	Summing \eqref{eq:b1} and \eqref{eq:b2} over
	\[
	m\in[-L/(k-k_Q),0]
	\]
	yields
	\begin{equation*}
		\sum_{-L/(k-k_Q)\le m\le 0}
		\sum_{\substack{\ell,n\le 0\\ \eqref{eq:BA_app},\,\eqref{eq:BC_app}}}
		B_{\ell,m,n}
		\lesssim
		\lambda^{-\rho_2}
		\sum_{-L/(k-k_Q)\le m\le 0}2^{\theta m}.
	\end{equation*}
	Using the geometric-sum principle,
	\[
	\sum_{-L/(k-k_Q)\le m\le 0}2^{\theta m}
	\lesssim
	\begin{cases}
		1,&\theta>0,\\
		L,&\theta=0,\\
		2^{\theta(-L/(k-k_Q))},&\theta<0.
	\end{cases}
	\]
	From \eqref{eq:theta_def}, $\theta=0$ is equivalent to the borderline condition \eqref{eq:borderlineA}, and in this case $\rho_1=\rho_2$.
	Moreover,
	\[
	2^{\theta(-L/(k-k_Q))}
	=
	\lambda^{-\theta/(k-k_Q)}
	=
	\lambda^{\rho_2-\rho_1}.
	\]
	Hence
	\[
	\sum_{-L/(k-k_Q)\le m\le 0}
	\sum_{\substack{\ell,n\le 0\\ \eqref{eq:BA_app},\,\eqref{eq:BC_app}}}
	B_{\ell,m,n}
	\lesssim
	\max\{\lambda^{-\rho_2},\,\lambda^{-\rho_1}(\log\lambda)^{\mathbf 1_{\eqref{eq:borderlineA}}}\}.
	\]
	Combining this with the subrange $m\le -L/(k-k_Q)$, for $k>2k_P$ we have
	\begin{equation}\label{eq:SB_kgt2kP}
		S_B\lesssim
		\max\{
		\lambda^{-\rho_1}(\log\lambda)^{\mathbf 1_{\eqref{eq:borderlineA}}},
		\,\lambda^{-\rho_2}
		\}.
	\end{equation}
	
	Putting together \eqref{eq:SB_kle2kP} and \eqref{eq:SB_kgt2kP} gives
	\begin{equation}\label{eq:SB_final_rig}
		S_B\lesssim
		\begin{cases}
			\max\{\lambda^{-\rho_1}(\log\lambda)^{\mathbf 1_{\eqref{eq:borderlineA}}},\,\lambda^{-\rho_2},\,\lambda^{-\rho_3}\},&k\neq 2k_P,\\[2pt]
			\lambda^{-\rho_3}\log\lambda,&k=2k_P.
		\end{cases}
	\end{equation}

	\subsection{Region III: $C\le A$ and $C\le B$}
	
	Let
	\[
	S_C:=\sum_{\substack{\ell,m,n\le0\\ C\le A,\,C\le B}}C_{\ell,m,n}.
	\]
	The constraints $C\le A$ and $C\le B$ are equivalent to
	\begin{equation}\label{eq:CA_app}
		(k-k_Q)m+k_Q n\ge -L,
	\end{equation}
	and
	\begin{equation}\label{eq:CB_app}
		(k-k_Q)m+(k_Q-k_P)n\ge (k-k_P)\ell,
	\end{equation}
	respectively.
	
	Unlike Regions I and II,
    Region III is more naturally analyzed after reversing the dyadic indices, i.e.
	\[
	u:=-\ell,\qquad v:=-m,\qquad w:=-n\qquad(u,v,w\in\Z_{\ge 0}).
	\]
	In these variables, \eqref{eq:CA_app}--\eqref{eq:CB_app} become
	\begin{equation}\label{eq:CAprime_app}
		(k-k_Q)v+k_Q w\le L,
	\end{equation}
	and
	\begin{equation}\label{eq:CBprime_app}
		(k-k_Q)v+(k_Q-k_P)w\le (k-k_P)u,
	\end{equation}
	while
	\[
	C_{\ell,m,n}=C_{-u,-v,-w}
	=
	2^{-L/2}\,2^{-u/2}\,
	2^{\frac{k-k_Q-1}{2}v}\,
	2^{\frac{k_Q-1}{2}w}.
	\]
	For fixed $(v,w)$ satisfying \eqref{eq:CAprime_app}, the inequality \eqref{eq:CBprime_app} forces
	\[
	u\ge \frac{(k-k_Q)v+(k_Q-k_P)w}{k-k_P}.
	\]
	Therefore the $u$-sum is a tail of a decaying geometric series:
	\[
	\sum_{u\ge \frac{(k-k_Q)v+(k_Q-k_P)w}{k-k_P}}
	2^{-u/2}
	\lesssim
	2^{-\frac{(k-k_Q)v+(k_Q-k_P)w}{2(k-k_P)}}.
	\]
	Hence
	\begin{align*}
		S_C
		&\lesssim
		2^{-L/2}
		\sum_{\substack{v,w\ge 0\\ \eqref{eq:CAprime_app}}}
		2^{\frac{k-k_Q-1}{2}v}
		2^{\frac{k_Q-1}{2}w}
		2^{-\frac{(k-k_Q)v+(k_Q-k_P)w}{2(k-k_P)}}\\
		&=
		2^{-L/2}
		\sum_{\substack{v,w\ge 0\\ (k-k_Q)v+k_Q w\le L}}
		2^{a_v v}\,2^{b_w w},
	\end{align*}
	where
	\[
	a_v:=\frac12\Bigl((k-k_Q-1)-\frac{k-k_Q}{k-k_P}\Bigr),
	\qquad
	b_w:=\frac12\Bigl((k_Q-1)-\frac{k_Q-k_P}{k-k_P}\Bigr).
	\]
	Note that $b_w>0$. For fixed $v$, the constraint \eqref{eq:CAprime_app} gives
	\[
	0\le w\le \frac{L-(k-k_Q)v}{k_Q}.
	\]
	Thus the $w$-sum is dominated by the maximal $w$:
	\[
	\sum_{0\le w\le \frac{L-(k-k_Q)v}{k_Q}}2^{b_w w}
	\lesssim
	2^{\frac{b_w}{k_Q}(L-(k-k_Q)v)}.
	\]
	Therefore
	\[
	S_C
	\lesssim
	2^{-L/2}\,
	2^{\frac{b_w}{k_Q}L}
	\sum_{0\le v\le L/(k-k_Q)}
	2^{(a_v-\frac{b_w}{k_Q}(k-k_Q))v}.
	\]
	Note that $b_w = \beta$ given in \eqref{eq:beta} yielding
	\[
	-\frac12+\frac{b_w}{k_Q} = -\frac12 + \frac{\beta}{k_Q}=-\rho_2,
	\]
	hence the prefactor equals $\lambda^{-\rho_2}$. Moreover, direct computation gives
	\[
	a_v-\frac{b_w}{k_Q}(k-k_Q)=-\theta,
	\]
	with the same $\theta$ as in \eqref{eq:theta_def}. Consequently,
	\begin{equation}\label{eq:SC_geomsum}
		S_C\lesssim
		\lambda^{-\rho_2}
		\sum_{0\le v\le L/(k-k_Q)}2^{-\theta v}.
	\end{equation}
	Evaluating the geometric sum in \eqref{eq:SC_geomsum} yields
	\[
	S_C\lesssim
	\begin{cases}
		\lambda^{-\rho_2},&\theta>0,\\
		\lambda^{-\rho_2}L=\lambda^{-\rho_1}L,&\theta=0\quad(\text{i.e.\ }\eqref{eq:borderlineA}),\\
		\lambda^{-\rho_2}\,2^{-\theta L/(k-k_Q)}=\lambda^{-\rho_1},&\theta<0,
	\end{cases}
	\]
	since $\theta=0\iff \rho_1=\rho_2\iff \eqref{eq:borderlineA}$ and
	\[
	2^{-\theta L/(k-k_Q)}
	=
	\lambda^{-\theta/(k-k_Q)}
	=
	\lambda^{\rho_2-\rho_1}.
	\]
	Thus
	\begin{equation}\label{eq:SC_final_rig}
		S_C\lesssim
		\max\{
		\lambda^{-\rho_1}(\log\lambda)^{\mathbf 1_{\eqref{eq:borderlineA}}},
		\,\lambda^{-\rho_2}
		\}.
	\end{equation}
	
	Since $S(\lambda)\le S_A+S_B+S_C$, combining \eqref{eq:SA_final_rig}, \eqref{eq:SB_final_rig}, and \eqref{eq:SC_final_rig} gives
	\[
	S(\lambda)\lesssim
	\begin{cases}
		\max\{\lambda^{-\rho_1}(\log\lambda)^{\mathbf 1_{\eqref{eq:borderlineA}}},\,\lambda^{-\rho_2},\,\lambda^{-\rho_3}\},&k\neq 2k_P,\\[2pt]
		\lambda^{-\rho_3}\log\lambda,&k=2k_P.
	\end{cases}
	\]

    If $k\neq2k_P$, the logarithmic loss can occur only in the
    $\lambda^{-\rho_1}$-term.
    Moreover, it contributes to the final estimate only when
    $\rho_1=\rho_*$, which is equivalent to the critical condition
    \eqref{eq:borderlineA}.
    Otherwise,
    \[
    \lambda^{-\rho_1}\log\lambda
    \lesssim
    \lambda^{-\rho_*}
    \]
    since $\rho_1>\rho_*$.

    When $k=2k_P$, we have $\rho_*=\rho_3$, and therefore
    \[
    S(\lambda)
    \lesssim
    \lambda^{-\rho_*}\log\lambda.
    \]
    Hence, in both cases,
    \[
    S(\lambda)
    \lesssim
    \lambda^{-\rho_*}(\log\lambda)^\gamma,
    \]
    where \(\gamma\) is given by \eqref{eq:gamma_def}.
	This completes the proof of Proposition~\ref{prop:triplesum}.

    \section*{Acknowledgement}

    This work is supported by the National Research Foundation of Korea~(NRF) grants funded by the Korea government~(MSIT)(RS-2026-25495323; C. Cho and C. W. Yang), (RS-2021-NR061906, RS-2024-00461749; J. B. Lee).

\end{document}